\documentclass[a4paper, twoside,12pt]{article}
\usepackage{fancyhdr}
\usepackage{fnpos}
 \usepackage[english]{babel}        
\usepackage[T1]{fontenc}          
\usepackage{graphicx}             
\usepackage{makeidx}
\usepackage{fancybox}
\usepackage{framed}
\usepackage{fancyhdr}
 \usepackage{pstricks,pst-plot,pstricks-add}
\usepackage[margin=1in]{geometry}
\usepackage{graphicx}
\usepackage{titlesec}
\usepackage{amsmath}
\usepackage{amsfonts}   
\usepackage{amssymb}    
\usepackage{amsthm}
\usepackage{dsfont}
\usepackage{mathtools}
\usepackage{verbatim}   
\usepackage{color}
\usepackage{listings}
\usepackage{graphicx}
\usepackage{amsmath}
\usepackage{amsmath,amssymb,color,inputenc,euscript,graphicx,psfrag}

\newtheorem{thm}{Theorem}[section]
\newtheorem{cor}[thm]{Corollary}
\newtheorem{lem}[thm]{Lemma}
\newtheorem{prop}[thm]{Proposition}

\numberwithin{equation}{section}

\renewcommand{\thefootnote}

 {  }{  }

\author { B\'echir Amri$^*$ and  Khawla Kerfef$^{**}$  }

\title{ A  Generalized discrete  Riesz  transforms}
\date{ }

\begin{document}
 \maketitle
\begin{center}
   $^*$Taibah University, College of Sciences, Department of Mathematics, P. O. BOX 30002, Al Madinah AL Munawarah, Saudi Arabia.\\
\textbf{ e-mail:} bechiramri69@gmail.com\\
  $^{**}$Universit\'{e} Tunis El Manar, Facult\'{e} des sciences de Tunis,\\ Laboratoire d'Analyse Math\'{e}matique
       et Applications,\\ LR11ES11, 2092 El Manar I, Tunisie.\\
     \textbf{ e-mail:} kh.karfaf@gmail.com
\end{center}
\begin{abstract}
In this paper, we  introduce a discrete Riesz transforms   associated with the  non-symmetric  trigonometric Heckman-Opdam polynomials of type $A_1$.  We prove that they can be extended to a bounded
operators on $\ell^p(\mathbb{Z})$,   $1<p<\infty$.
\footnote{\noindent\textbf{Key words and phrases:} Heckman-Opdam polynomials,  discrete Riesz transforms.  \\
\textbf{Mathematics Subject Classi
fication}.Primary 39A12; Secondary  47B38. }
 \end{abstract}
\section{Introduction}
Fourier series constitute a fundamental tool in the study of periodic functions.
It is from this concept that a  branch of mathematics known as harmonic analysis was  developed. There are
many generalizations of Fourier series that have proved to be useful and  are all special cases of decompositions over an orthonormal basis
 of an inner product space.  In this work,  we consider the Fourier  series  generated by the non-symmetric  trigonometric Heckman-Opdam polynomials of type $A_1$ which  contains  as special case (k=0) the classical Fourier  series. Generally,  the Non-symmetric Heckman-Opdam polynomials are     family of orthogonal
polynomials  associated with a root system, introduced by  Opdam \cite{O1} as eigenfunctions of Cherednik operators.
In particular,  for  root system of type  $A_1$ these  polynomials are reduced  to the non-symmetric  ultraspherical or Gegenbauer polynomials.
\par The aim of this paper is to establish the  $\ell^p$-boundedness of the discrete  Riesz transforms associated with the  non-symmetric Heckman-Opdam polynomials.
We prove that they  are a   Calder\'{o}n-Zygmund  operators on the discrete homogeneous space $\mathbb{Z}$. The key
point for this consideration is  a discrete version of H\"{o}rmander type conditions given in \cite{B}, and closely connect with  a  continuous  version  given in \cite{BA}. We recall   that   discrete   Riesz transforms  for  ultraspherical   polynomials are recently studied in \cite{B} and  in \cite{AA} for Jacobi polynomials.
  \section{Non-symmmetric Heckmann-Opdam polynomials  of type $A_1$ }
\subsection{Definitions and properties}
For the theory of   Heckmann-Opdam polynomials associated to general root systems,  we refer to   \cite{HO2}.
\par Let $k\geq 0$, the Cherednick operator of type $A_1$ is defined by
\begin{eqnarray}\label{T}
T^{k}(f)(x)=f'(x)+ 2k\;\frac{ f(x)-f(-x)}{1-e^{-2x}}-   k f(x), \qquad f\in C^1(\mathbb{R}).
 \end{eqnarray}
 We introduce  the  weight function
$$\delta_k(x)=|2\sin x |^{2k}   $$
and the inner product on $L^2([0,2\pi),\delta_k(x)dx)$
$$(f,g)_k=\frac{1}{2\pi}\int_{0}^{2\pi}f(x)\overline{g(x)}\delta(x)dx,$$
with its associate norm $\|.\|_k$.
  \par Defining the   partial ordering $\triangleleft$ on $\mathbb{Z}$ as follows
\begin{equation*}
  j\triangleleft n  \Leftrightarrow \left\{
    \begin{array}{ll}
    |j| < |n|\;\text{and} \;|n|-|j|\in 2\mathbb{Z}^+   , & \hbox{} \\
    or   , & \hbox{} \\
      |j| = |n|\;\text{and} \; n <j . & \hbox{.}
    \end{array}
  \right.
\end{equation*}

The non-symmetric  Heckman-Opdam polynomials $E_n^k$, $n\in \mathbb{Z}$  are  defined by the following conditions:
\begin{eqnarray}\label{def}
 (a)&&    E_n^{k}(x)=e^{nx}+ \sum_{j\triangleleft n }c_{n,j}\; e^{jx}\qquad\qquad\qquad\qquad\qquad\qquad\qquad\qquad\qquad
 \\ (b) && (E_n^{k}(ix),e^{ijx})_{k}=0,\quad \text{for any } \; j\triangleleft n.
\end{eqnarray}
The  polynomials $E_n^k$ diagonalize simultaneously the  Cherednik operators  and
\begin{equation}\label{TE}
  T^{k} ( E_n^{k})=  \widetilde{n} \; E_n^{k}, \qquad n\in \mathbb{Z},
\end{equation}
 where $ \tilde{n}=n+k $ if $n\geq 0$ and  $\tilde{n}=n-k$ if $n<0$.
Clearly,
\begin{equation}\label{13}
 E_0^k(x)=1,\quad \text{and}\quad E_1^k(x)=e^x.
\end{equation}
 A a consequence the trigonometric polynomials
$\{ E_n^{k}(ix) \}_{n\in \mathbb{Z}}$    is an orthogonal basis of   $L^2([0,2\pi),\delta_k(x)dx)$.\\
 \par The Heckman-Opdam- Jacobi  polynomials  $P_n^k$, $n\in \mathbb{Z}_+$ is given  by
  $$P_n^{k}(x) = E_n^{k}(x)+E_n^{k}(-x).$$
and they satisfy
$$ (T^k)^2P_n^k=(n+k)^2 P_n^k.$$
In particular  if we denote $L_k$ the differential operator
$$L_k(f)=f''+2k\; \frac{1+e^{-2x}}{1-e^{-2x}}\;f'(x)+k^2f(x)$$
then $P_n^k$ is the eigenfunction of $L_k$, namely that
$$L_k(P_n^k)=(n+k)^2P_n^k.$$
The Heckman opdam- Jacobi  polynomials is expressed via the hypergeometric function $_2F_1$
by
$$P_n^k(x)=P_n^k(0)\;_2F_1(n+2k,-n,k+1/2,\sinh^2(x/2))$$
where for all $n\in \mathbb{Z}^+$
$$ P_n^k(0)=2 E_n^{k}(0)= \frac{\Gamma(k)\Gamma(n+2k)}{\Gamma(2k)\Gamma(n+k)}.$$
Noting also that
$$  E_{-n}^k(0)= \frac{\Gamma(k)\Gamma(n+2k+1)}{2\Gamma(2k)\Gamma(n+k+1)}, \quad n>0.$$
 We have  following relationship
\begin{equation*}\label{jacobi}
   E_n^k(x)=E_n^k(0)\Big\{ P_{|n|}^k( x)+\frac{\tilde{n}+2k}{2k+1}\sinh x \;P_{|n|-1}^{k+1}(x)\Big\}, \quad n\in\mathbb{Z}.
\end{equation*}
 \par It is also  worth mentioning that  the non-symmetric  Heckman-Opdam polynomials $E_n^k$ are closely  related to non-symmetric
Jack polynomials  see \cite{Sah}.
If  $\lambda=(\lambda_1,\lambda_2)\in \mathbb{Z}\times \mathbb{Z} $ and $\mathcal{N}_{\lambda }^k$ the correspondent non-symmetric
Jack polynomials then
$$E_{\lambda_2-\lambda_1}^k= \mathcal{N}_{\lambda}^k(e^{-x},e^{x}).$$
 An important  result due to  Sahi \cite{Sah} states that the coefficients  $c_{n,j} $ in  (\ref{def}) are all nonnegative.
\par The  upcoming propositions are  inspired by the work of \cite{Sah}.
\begin{prop}
 For all $n\in \mathbb{Z}$, we have
\begin{equation}\label{12}
   E_{n+1}^k(x)=e^xE_{-n}^k(-x).
\end{equation}
\end{prop}
\begin{proof}
In view of (\ref{13}) this identity is true    for $n=-1,0$. Let $n\in \mathbb{Z}$, $n\neq -1,0$. Put
 $H_n^k(x)= e^xE_{-n}^k(-x)$, it is  enough to check that $ T^k(H_n)=(\widetilde{n+1})H_n$. We have
\begin{eqnarray*}
&&T^k(H_n)(x)
\\ &&\quad =e^xE_{-n}^k(-x)-e^x(E_{-n}^k)'(-x)+2k \; \frac{e^xE_{-n}^k(-x)-e^{-x}E_{-n}^k(x)}{1-e^{-2x}}-ke^xE_{-n}^k(-x)
\\&&\quad= -e^x \left((E_{-n}^k)'(-x)+2k\;\frac{E_{-n}^k(-x)-E_{-n}^k(x)}{1-e^{2x}} -k E_{-n}^k(-x) \right)+  e^xE_{-n}^k(-x)
\\&&\quad= ( 1- (\widetilde{-n}))e^xE_{-n}(-x)=(\widetilde{n+1})e^xE_{-n}(-x)=(\widetilde{n+1})H_n(x).
\end{eqnarray*}
Then (\ref{12})  follows  by comparing the highest coefficient.
\end{proof}
\begin{prop}
 For $n \in \mathbb{Z}^+$, $n\neq 0$, we have
\begin{equation}\label{45}
 E_{-n}^k(x)=E_{n}^k(-x) +\frac{k}{n+k}E_n^k(x).
\end{equation}
  \end{prop}
\begin{proof}
Observe that  for $n\geq 1$ we have $T^k(E_n(-x))=-2kE_n(x)- (n+k)E_n(-x)$. Then
\begin{eqnarray*}
  T^k\left( E_n^k(-x)+\frac{k}{n+k}E_n^k(x)\right)&=&-2kE_{n}(x)- (n+k)E_n(-x)+ kE_{n}(x)
\\&=&-kE_{n}(x)-(n+k)E_n^k(-x)
\\&=&-(n+k)\left(E_n^k(-x)+\frac{k}{n+k}E_n^k(x)\right).
 \end{eqnarray*}
We conclude the  (\ref{45})  by comparing the highest coefficient.
\end{proof}
From (\ref{45}) one can deduce the  following identity.
\begin{cor}
  For $n \in \mathbb{Z}^+$, we have
\begin{equation}\label{56}
 \left(1-\frac{k^2}{(n+k)^2}\right)E_n^k(x)=E_{-n}^k(-x)-\frac{k}{n+k}\;E_{-n}^k(x).
\end{equation}
\end{cor}
  \begin{prop}
  For all $n  \in \mathbb{Z}^+$, we have
 $$ \|E_{n+1}^k\|_k^2=\|E_{-n}^k\|_k^2= n!\; \frac{\Gamma(n+2k+1)}{\Gamma(n+k+1)^2}. $$
\end{prop}
\begin{proof}
In view of  (\ref{13}), we write
$$  E_n^k(-x)= E_{-n}^k(x)-\frac{k}{n+k}E_n^k(x), \quad n\geq 1\quad.$$
  Since $E_n^k$ and $E_{-n}^k$ are orthogonal, then
$$\|E_n^k\|_k^2= \|E_{-n}^k\|_k^2 +\frac{k^2}{(n+k)^2}\|E_n^k\|_k^2$$
 and
$$\|E_{n+1}^k\|_k^2= \|E_{-n}^k\|_k^2 =\frac{n(n+2k)}{(n+k)^2} \|E_n^k\|_k^2 . $$
In addition  it not hard to see that
 $$ \|E_1^k\|_k^2=\|E_{0}^k\|_k^2=   \frac{\Gamma(2k+1)}{\Gamma(k+1)^2}.$$
Therefore the desired formula follows.
\end{proof}
In what follows we set
  $$ \mathcal{E}_n^k=\frac{E_n^k}{\|E_n^k\|_k}.$$
From (\ref{56}) we get
$$
   \frac{\sqrt{n(n+2k)}}{n+k}\; \mathcal{E}_n^k(x)= e^{-x}\mathcal{E}_{n+1}^k-\frac{k}{n+k}\;\mathcal{E}_{-n}^k(x), \quad n\geq0,
$$
and so,
\begin{equation}\label{1}
  e^{-x} \mathcal{E}_{n+1}^k=\frac{\sqrt{n(n+2k)}}{n+k}\; \mathcal{E}_n^k(x)+ \frac{k}{n+k}\;  \mathcal{E}_{-n}^k(x), \quad \quad n\geq0.
\end{equation}
However  from  (\ref{12}) and  (\ref{45})  we  have  that
$$ \mathcal{E}_{-n}^k(x)= e^{-x}\mathcal{E}_{-n+1}(x) +\frac{k}{n+k}\;\mathcal{E}_n^k(x), \quad n\geq 1,$$
which implies that
$$ \frac{\sqrt{n(n+2k)}}{n+k}\; \mathcal{E}_{-n}^k(x)= e^{-x}\mathcal{E}_{-n+1}^k+\frac{k}{n+k}\;\mathcal{E}_n^k(x),\quad n\geq 1$$
and  then
\begin{equation}\label{2}
  e^{-x}\mathcal{E}_{n+1}^k=\frac{\sqrt{-n(-n+2k)}}{-n+k}\mathcal{E}_n^k(x)-\frac{k}{-n+k}\;\ \mathcal{E}_{-n}^k(x),\quad n\leq - 1.
\end{equation}
Now putting
\begin{equation}\label{alpha}
   \alpha_n=\frac{\sqrt{|n|(|n|+2k)}}{|n|+k} \qquad\text{and}\qquad \beta_n= \varepsilon(n)\frac{k}{|n|+k},\quad n\in \mathbb{Z},
\end{equation}
with
 $$\varepsilon(n)= \left\{
     \begin{array}{ll}
       1, & \hbox{$n\geq 0$}\\
      -1, & \hbox{$n < 0$}
     \end{array}
   \right.
 $$
We then   state  the following.
 \begin{prop} For all $n\in \mathbb{Z}$,  we have
\begin{equation}\label{eE}
   e^{-x}\mathcal{E}_{n+1}^k=\alpha_n \;\mathcal{E}_n^k +\beta_n \;\mathcal{E}_{-n}^k=  \alpha_n \;\mathcal{E}_n^k -\beta_{-n} \;\mathcal{E}_{-n}^k.
\end{equation}
   \end{prop}
  Next we define on  $\mathbb{C}^{\mathbb{N}}$  the operators $ \Lambda_k$ and $\Lambda_k^*$ by:
\begin{eqnarray*}
  \Lambda_k(f)(n)&= & =\alpha_nf(n+1)+ \beta_{-n} f(-n+1)-f(n),
 \\ \Lambda_k^*(f)(n)&=  & \alpha_{n-1}f (n-1)+ \beta_{n-1} f(-n+1)-f(n).
\end{eqnarray*}
  and  the generalized  discrete  Laplace operator  by
$$\Delta_k= -\Lambda_k^*\Lambda_k .$$
 It follows that
$$\Delta_k(f)(n)=\alpha_n f(n+1)+\alpha_{n-1}f(n-1)-2f(n)-(\beta_n-\beta_{n-1})f(-n+1), \quad f\in\ell^2(\mathbb{Z}).$$
In the  case $k=0$, we have
$$\Lambda_0f(n)=f(n+1)-f(n),\quad  \Lambda_0^*f(n)=f(n-1)-f(n)$$
and
$$\Delta_k(f)(n)=f(n+1)-2f(n)+f(n-1). $$
 Let us   introduce  the operator
$$ \widetilde{T}^k(u)(x)=u'(x)+2ki\; \frac{u(x)-u(-x)}{1-e^{-2ix}}; \qquad u\in C^1(\mathbb{R}).$$
In view of (\ref{T}) and (\ref{TE}) we have
\begin{prop}\label{Tt}
  For all $n\in \mathbb{ Z}^+$ we have
$$ \widetilde{T}^k( \mathcal{E}_n^k(i.))(x)=i(\widetilde{n}+k)\mathcal{E}_n^k(ix).$$
\end{prop}
\subsection{ Generalized  discrete Fourier transform}
 For  $f\in \ell^2(\mathbb{Z})$ we define the generalized discrete fourier transform of $f$ by
$$\mathcal{F}_k(f)(x)=\sum_{n\in \mathbb{Z}} f(n)\mathcal{E}_n^k(ix), \quad n\in  \mathbb{Z}.$$
As $\{\mathcal{E}_n^k(ix),\; n\in \mathbb{Z}\}$ is an orthonormal basis of $L^2([0,2\pi),  \delta_k(x)dx)$,
then we can state,
\begin{thm}
  $\mathcal{F}_k$   is an isometric isomorphism  from   $\ell^2(\mathbb{Z})$ onto $L^2([0,2\pi),  \delta_k(x)dx).$
Precisely we have
$$\|\mathcal{F}_k(f)\|_k^2=\sum_{n\in \mathbb{Z}}|f(n)|^2$$
and for $n\in \mathbb{Z}$,
$$f(n)=\frac{1}{2\pi}\int_{0}^{2\pi}\mathcal{F}_k(f)(x)\overline{\mathcal{E}_n^k(ix)}\delta_k(x)dx.$$
\end{thm}
 It's not hard to see that
\begin{eqnarray*}
&&\mathcal{F}_k(\Lambda_k(f)) (x)=(e^{-ix}-1) \mathcal{F}_k(\Lambda_k(f)) (x), \quad
  \mathcal{F}_k(\Lambda_k^*(f)) (x) = (e^{ix}-1) \mathcal{F}_k(\Lambda_k(f)) (x)
 \end{eqnarray*}
 and
$$\mathcal{F}_k(\Delta_k(f))(x)= - 4\sin^2(x/2)\mathcal{F}_k(f)(x).$$
Clearly    $- \Lambda_k  $   is positive  self-adjoint operator on
 $\ell^2 (\mathbb{Z})$ and    generates a $C_0$-semigroup  $e^{t\Lambda_k}$, $t\geq 0$  given   by
$$\mathcal{F}_k(e^{t\Lambda_k}(f))(x)= e^{-4t\sin^2(x/2)}\mathcal{F}_k(f)(x)  $$
 It can be   written as
$$e^{t\Lambda_k}(f)(n)= \sum_{m\in \mathbb{Z} } H_t(n,m)f(m),\quad n\in \mathbb{Z},$$
with
$$H_t(n,m)=\frac{1}{2\pi}\int_{0}^{2\pi}e^{-4t\sin^2(x/2)}\mathcal{E}_m^k(ix)\mathcal{E}_n^k(-ix) \delta_k(x)dx,$$
which  will be  called  the generalized discrete heat kernel.
 \par We now introduce the  generalized discrete Riesz Transforms given as  a multiplier operators by
 \begin{equation}\label{riesz}
    \mathcal{F}_k(R(f))(x)=-ie^{-ix/2}\mathcal{F}_k(f)(x),\quad  \mathcal{F}_k(R^*(f))(x)=ie^{ix/2}\mathcal{F}_k(f)(x), \quad x\in [0,2\pi].
 \end{equation}
They can be  written as
$$R=\Lambda_k(-\Delta_k)^{-1/2}, \quad \text{and} \quad R^*=\Lambda_k^*(-\Delta_k)^{-1/2}.$$
 It follows easily that   $R$ and $R^*$ are a bounded operators on $\ell^2(\mathbb{Z})$ and they are a kernel operators, since we have
\begin{eqnarray*}
R(f)(n)&=&-i\int_{0}^{2\pi}e^{-ix/2}\mathcal{F}_k(f)(x)\mathcal{E}_n^k(-ix)\delta_k(x)dx
\\&=&\frac{1}{2\pi i}\sum_{n\in \mathbb{Z}}f(m)\int_{0}^{2\pi}e^{-ix/2}\mathcal{E}_m^k(ix)\mathcal{E}_n^k(-ix)\delta_k(x)dx.
\\&=&\sum_{m\in \mathbb{Z}}\mathcal{R}(n,m)f(m),
\end{eqnarray*}
where
\begin{eqnarray*}
  \mathcal{R}(n,m)&=&\frac{1}{2\pi i} \int_{0}^{2\pi}e^{-ix/2}\mathcal{E}_m^k(ix)\mathcal{E}_n^k(-ix)\delta_k(x)dx.
\\&=&  \frac{1}{2\pi i} \int_{-\pi}^{\pi} h(x) \mathcal{E}_m^k(ix)\mathcal{E}_n^k(-ix)\delta_k(x)dx.
 \end{eqnarray*}
with $h(x)=sign(x)e^{-ix/2}$. Similarly
\begin{eqnarray*}
R^*(f)(n)=\sum_{m\in \mathbb{Z}}\mathcal{R}^*(n,m)f(m)
\end{eqnarray*}
where
\begin{eqnarray*}
  \mathcal{R}^*(n,m)=-\frac{1}{2\pi i} \int_{0}^{2\pi}e^{ix/2}\mathcal{E}_m^k(ix)\mathcal{E}_n^k(-ix)\delta_k(x)dx.
  \end{eqnarray*}

\section{$\ell^p$-boundedness of the generalized  Riesz transform}
The main tool  to  study  the $\ell^p$-boundedness of a discrete    integral  operator  is a  discrete version  of  an $\ell^p$- theorem   adopted in  Dunkl theory  by the  first author \cite{BA1,BA}. The ideas for this material come from   \cite{B,AA}.
\subsection{ $\ell^p$-Theorem  for a  discrete integral operators }
\par Let us begin by remember  the following continuous version
\begin{thm}
   Let $\mathcal{K} $ be a measurable function on $\mathbb{R}^2\setminus\{(x,y);\;|x|\neq |y|\} $
and $S$ be a bounded operator on $L^2(\mathbb{R},dx)$  such that
   $$T(f)(x)=\int_{\mathbb{R}}K(x,y)f(y)dy, $$
for all a.e. $x\in \mathbb{R}$, such that $x,-x\notin supp(f)$. If $K$ satisfies the following H\"{o}rmander type condition
\begin{eqnarray*}
\sup_{y,y'\in \mathbb{R}} \int_{||x|-|y||>2|y-y'|}\Big(|K(x,y)-\mathcal{K}(x,y')|+|K(y,x)-K(y',x)|\Big)dx<\infty
\end{eqnarray*}
then $T$ can be extended to bounded operator from $L^p(\mathbb{R},dx)$ onto itself for $1<p<\infty$.
\end{thm}
The proof of this theorem only requires some minor modifications of the known classical ones and can be easily  adopted to discrete analogues version
by considering the metric space $\mathbb{Z}$ with the  counting measure   which is   invariant by $-Id$.
We state the following,
\begin{thm}\label{ttmm}
   Let $K $ be a measurable function on $ \mathbb{Z}\times \mathbb{Z}\setminus\{(n,m);\;|m|\neq |n|\} $
and $T$ be a bounded operator on $\ell^2(\mathbb{Z})$  such that for a compact support fonction $f\in \ell^(\mathbb{Z})$,
   $$T(f)(n)= \sum_{m\in \mathbb{Z}} K(n,m)f(m), $$
for all   $n\in \mathbb{Z}$, such that $f(n)=f(-n)=0$. If $\mathcal{K}$ satisfies the following H\"{o}rmander type condition
\begin{eqnarray}\label{Hor}
\sup_{m,\ell\in \mathbb{Z}} \sum_{||n|-|m||>2|m-\ell|}\Big(|K(n,m)-K(n,\ell)|+|K(m,n)-K(\ell,n)|\Big) <\infty.
\end{eqnarray}
then $T$ can be extended to a bounded operator from $\ell^p(\mathbb{Z})$ onto itself, for all$1<p<\infty$.
\end{thm}
Our strategy in applying this theorem inspired by  \cite{B}, through the following  result,
\begin{thm}\label{3.3}
  Suppose that $T$ is a linear and bounded operator on $\ell^2(\mathbb{Z})$ and such that there exists a function
$K:\mathbb{Z}\times \mathbb{Z}\setminus \{(m,n),\; |m|\neq |n|\}$ such that for every $f=(f(n))_{n\in \mathbb{Z}}$
\begin{equation}\label{Tf}
 T(f)(n)=\sum_{m\in \mathbb{Z}}K(n,m)f(m)
\end{equation}
for $n\in \mathbb{Z}$ such that $f(n)=f(-n)=0$. In addition we assume that $K$ satisfies the following H\"{o}rmander conditions
\begin{eqnarray*}
&& (i) \quad|K(n,m)|\leq \frac{C}{| n-m|},
\\&& (ii) \quad|K(n,m)-K(n,\ell)|+\quad |K(m,n)-K(\ell,n)|\leq C\; \frac{ |n-\ell|}{(|m|-|n|)^2},
  \end{eqnarray*}
for all $m,n,\ell$ such that $$ ||m|-|n||>2|n-\ell| \quad\text{and}\quad  \frac{|n|}{2}\leq |m|,|\ell|\leq \frac{3|n|}{2}.$$
Then   $T$ can be extended to bounded operator from $\ell^p(\mathbb{Z})$ onto $\ell^p(\mathbb{Z})$,  for  all $1<p<\infty$.
\end{thm}
\begin{proof}
We   apply the  same arguments that used  in   \cite{B}.
  \par For  $n\in \mathbb{Z}$, we let  $$W_n= \{m\in \mathbb{Z}; \; |n|/2\leq  |m|\leq 3|n|/2\}.$$  Define the operators
$$T_{glob}(f)(n)=T(\chi_{\mathbb{Z}\setminus W_n}f)(n), \quad n\in \mathbb{Z}$$
and
$$T_{loc}(f)=T(f)-T_{glop}(f),$$
for   compactly  supported $f = (f(n))_{n\in \mathbb{Z}}$.
We  first check the boundedness of  $T_{glob}$. According to $(\ref{Tf})$, one can    write
\begin{eqnarray*}
 T_{glob}(f)(n)&= &\sum _{m\in \mathbb{Z}\setminus W_n}K(n,m)f(m)=\sum _{|m|\leq |n|/2}K(n,m)f(m)+\sum _{|m| \geq 3|n|/2}K(n,m)f(m).
\end{eqnarray*}
By using $(i)$ we have that
\begin{eqnarray*}
 | T_{glob}(f)(n)|&\leq&  C\left\{\frac{ 1}{|n|} \sum _{|m|\leq |n| } |f(m)|+\sum _{|m| \geq |n| } \frac{|f(m)|}{|m|}\right\}
\\&\leq &  C\Big\{ H_0(\widetilde{f})(|n|) +H_1(\widetilde{f})(|n|)\Big\}
\end{eqnarray*}
where $\widetilde{f}(n)=|f(n)|+|f(-n)|$,  $H_0$ and $H_1$ are the discrete  Hardy operators   given on $\mathbb{C}^{\mathbb{Z}^+}$ by
$$H_0(a)(n)=\frac{1}{n}\sum_{m=0}^n a(n), \quad H_1(a)(n)=\sum_{m\geq n}  \frac{a(n)}{n} .$$
Since it's known that these operators are bounded on $\ell^p(\mathbb{N})$ for $1<p<\infty$,    then $ T_{glob}$  can be extended to bounded operator on $\ell^p(\mathbb{\mathbb{Z}})$.
 \par Next we focus on $T_{loc}$. Define the kernel $\widetilde{K}$ by
$$\widetilde{K}(n,m)=\chi_{W_n}(m)K(n,m); \qquad n,m\in \mathbb{Z}; \; |n|\neq |m|.$$
We can write
$$ T_{loc}(f)(n) =\sum _{m\in \mathbb{Z} }\widetilde{K}(n,m)f(m)$$
for compactly supported function $f$ with $f(n)=f(-n)=0$.
We will prove that $\widetilde{K}$ satisfies the H\"{o}rmander type condition (\ref{Hor}).
\par Let $m,\ell \in \mathbb{Z}$ with $|m|<|\ell|$. We have
\begin{eqnarray*}
&& \sum_{||n|-|m||\geq 2|m-\ell|}|\widetilde{K}(n,m)-\widetilde{K}(n,\ell)|= \qquad\qquad\qquad\qquad\qquad\qquad\qquad\qquad\qquad
\\&& \qquad\qquad\underset{ |n| < 2|\ell| /3}{\sum_{||n|-|m||\geq 2|m-\ell|}}|\widetilde{K}(n,m)-\widetilde{K}(n,\ell)|
+ \underset{|n|\geq 2|\ell|/3 }{\sum_{||n|-|m||\geq 2|m-\ell|}}|\widetilde{K}(n,m)-\widetilde{K}(n,\ell)|.
\end{eqnarray*}
Making  use of the assertion $(i)$
\begin{eqnarray*}
   \underset{|n|< 2|\ell| /3}{\sum_{||n|-|m||\geq 2|m-\ell|}}|\widetilde{K}(n,m)-\widetilde{K}(n,\ell)|& = &\underset{   |n| < 2|\ell|/3}{\sum_{||n|-|m||\geq 2|m-\ell|}|}\widetilde{K}(n,m)|
\\&\leq&  C\;\underset{ 2|m|/3\leq |n|\leq 2|\ell|/3}{\sum_{||n|-|m||\geq 2|m-\ell|}}\frac{1}{||n|-|m||}
\\&\leq&  \frac{C}{|\ell|-|m|}\sum_{2|m|/3\leq |n|\leq 2|\ell|/3}\leq C
\end{eqnarray*}
and  by   $(ii)$
\begin{eqnarray*}
 \underset{|n|\geq 2|\ell|/3 }{  \sum_{||n|-|m||\geq 2|m-\ell|}}|\widetilde{K}(n,m)-\widetilde{K}(n,\ell)|&\leq& C \; |m-\ell|\sum_{||n|-|m||\geq 2|m-\ell|}
\frac{1}{(|n|-|m|)^2}
\\ &\leq& C |m-\ell| \sum_{j\geq 2|m-\ell|}\frac{1}{j^2}\leq C.
\end{eqnarray*}
Hence,  we conclude that
\begin{eqnarray*}
   \sum_{||n|-|m||\geq 2|m-\ell|}|\widetilde{K}(n,m)-\widetilde{K}(n,\ell)| \leq C
\end{eqnarray*}
Let us now prove that
\begin{eqnarray}
  \sum_{||n|-|m||\geq 2|m-\ell|}|\widetilde{K}(m,n)-\widetilde{K}(\ell,n)| \leq C
\end{eqnarray}
  Noting  first that if $||n|-|m||\geq 2|m-\ell|$  then
\begin{equation}\label{nm}
 ||n|-|\ell||\geq  \frac{||n|-|m||}{2}.
\end{equation}
Begin with  the case $3|m|<|\ell|$. Then  $W_m\cap W_\ell=\emptyset$ and
by using  (i) and (\ref{nm}) we have that
\\ \\ $\displaystyle{ \sum_{||n|-|m||\geq 2|m-\ell|}|\widetilde{K}(m,n)-\widetilde{K}(\ell,n)| }$
\begin{eqnarray*}
&& =  \underset{n\in W_m }{\sum_{||n|-|m||\geq 2|m-\ell|}}|\widetilde{K}(m,n)-\widetilde{K}(\ell,n)|
 + \underset{n\in W_\ell }{\sum_{||n|-|m||\geq 2|m-\ell|}}|\widetilde{K}(m,n)-\widetilde{K}(\ell,n)|
\\ && \leq C \; \Big(\underset{n\in W_m}{\sum_{||n|-|m||\geq 2|m-\ell|}}\frac{1}{||n|-|m||}+
\underset{n\in  W_\ell}{\sum_{||n|-|m||\geq 2|m-\ell|}}\frac{1}{||n|-|m||}\Big)
\\&&\leq  \frac{C}{|\ell|} \Big( \sum_{n\in W_m}+\sum_{n\in W_\ell} \Big) \leq C\;\frac{(|m|+|\ell|)}{|\ell|}\leq C.
\end{eqnarray*}
 Therefore, when $3|m|<|\ell|$,
\begin{eqnarray}
  \sum_{||n|-|m||\geq 2|m-\ell|}|\widetilde{K}(m,n)-\widetilde{K}(\ell,n)| \leq C.
\end{eqnarray}
Next,  we assume that $|\ell|\leq 3|m|$. If further $|\ell| < 9|m|/4$, then  we have $2|\ell|/3< 3|m|/2$ and we write
\begin{eqnarray*}
 &&\sum_{||n|-|m||\geq 2|m-\ell|}|\widetilde{K}(m,n)-\widetilde{K}(\ell,n)|\qquad\qquad\qquad\qquad\qquad
\qquad\qquad\qquad\qquad\qquad\qquad\qquad\qquad\qquad\\&&\qquad\qquad\qquad\qquad\quad
\\&&= \underset{|m|/2\leq |n|< 2|\ell|/3}{\sum_{||n|-|m||\geq 2|m-\ell|}}|\widetilde{K}(m,n)-\widetilde{K}(\ell,n)|+
\underset{2|\ell|/3\leq |n|\leq  3|m|/2}{\sum_{||n|-|m||\geq 2|m-\ell|}}|\widetilde{K}(m,n)-\widetilde{K}(\ell,n)|
\\ &&+\underset{3|m|/2 < |n|\leq  3|\ell|/2}{\sum_{||n|-|m||\geq 2|m-\ell|}}|\widetilde{K}(m,n)-\widetilde{K}(\ell,n)|.
\end{eqnarray*}
Since
 $$\{n\in \mathbb{Z};\;2|\ell|/3 \leq |n|\leq 3|m|/2\}\subset \{n \in \mathbb{Z};\; |n|/2   \leq |\ell|,|m|< 3|n|/2\}$$
so, using $(ii)$   we get
 \begin{eqnarray*}
 && \underset{2|\ell|/3\leq |n|\leq 3|m|/2}{\sum_{||n|-|m||\geq 2|m-\ell|}}|\widetilde{K}(m,n)-\widetilde{K}(\ell,n)| \leq C.
  \sum_{||n|-|m||\geq 2|m-\ell|} \frac{|m-\ell|}{(|n|-|m|)^2}\leq C
 \end{eqnarray*}
 According to $(i)$ and (\ref{nm})  we obtain that
\begin{eqnarray}
 \underset{3|m|/2 < |n| \leq 3|\ell|/2}{\sum_{||n|-|m||\geq 2|m-\ell|}}|\widetilde{K}(m,n)-\widetilde{K}(\ell,n)|
&\leq&  C \underset{3|m|/2\leq |n|\leq 3|\ell|/2}{\sum_{||n|-|m||\geq 2|m-\ell|}}\frac{1}{||n|-|m||}\nonumber
\\&\leq& C\; \sum_{3|m|/2\leq |n|\leq 3|\ell|/2}\frac{1}{|m-\ell|}
\leq C.\label{02}
\end{eqnarray}
 and
\\ \\ $\displaystyle{\underset{|m|/2\leq |n|< 2|\ell|/3}{\sum_{||n|-|m||\geq 2|m-\ell|}}|\widetilde{K}(m,n)-\widetilde{K}(\ell,n)|}$
\begin{eqnarray}
 && \nonumber\qquad\qquad\qquad\leq C\; \underset{|m|/2\leq |n|\leq 2|\ell|/3}{\sum_{||n|-|m||\geq 2|m-\ell|}}\frac{1}{||n|-|m||}
\\&&\nonumber\qquad\qquad\qquad\leq C\;\underset{|m|/2\leq |n|\leq 2|m|/3}{\sum_{||n|-|m||\geq 2|m-\ell|}}\frac{1}{||n|-|m||}+ C\;\underset{2|m|/3\leq |n|\leq 2|\ell|/3}{\sum_{||n|-|m||\geq 2|m-\ell|}}\frac{1}{||n|-|m||}
\\&&\qquad\qquad\qquad\leq  \frac{ C}{|m|}\sum_{|m|/2\leq |n|\leq 2|m|/3}+ C\;
\sum_{2|m|/3\leq |n|\leq 2|\ell|/3}\frac{1}{ |m-\ell|}\leq C.\label{01}
\end{eqnarray}
Now suppose that $\ell  > 9|m|/4$. In this case we have $3|m|/2<2|\ell|/3$ and
\begin{eqnarray*}
 &&\sum_{||n|-|m||\geq 2|m-\ell|}|\widetilde{K}(m,n)-\widetilde{K}(\ell,n)|\qquad\qquad\qquad\qquad\qquad
\qquad\qquad\qquad\qquad\qquad\qquad\qquad\qquad\qquad\\&&\qquad\qquad\qquad\qquad\quad
\\&&\leq  \underset{|m|/2\leq |n|< 2|\ell|/3}{\sum_{||n|-|m||\geq 2|m-\ell|}}|\widetilde{K}(m,n)-\widetilde{K}(\ell,n)|+
\underset{3|m|/2\leq |n|\leq  3|\ell|/2}{\sum_{||n|-|m||\geq 2|m-\ell|}}|\widetilde{K}(m,n)-\widetilde{K}(\ell,n)|\leq C,
 \end{eqnarray*}
which is an immediate  consequence of  (\ref{01}) and  (\ref{02}). This achieves the proof of Theorem \ref{3.3}.
 \end{proof}
\subsection{ Boundedness of the Riesz Transforms}
 We come now to  the $\ell^p$-boundedness of the Riesz transform, we shall prove that the kernel  $\mathcal{R}(n,m)$ satisfies the condition of
 Theorem \ref{3.3}.
\begin{lem}\label{l1}  For all $m,n\in \mathbb{Z}$, we have
$$(\widetilde{m}-\widetilde{n}-1/2)\mathcal{R}(n,m)=\frac{k}{\pi}\int_{-\pi}^{\pi}e^{-ix}\mathcal{E}_m(ix)\mathcal{E}_n^k(ix) \cos (x/2)\delta_{k-1/2}(x)dx.$$
  \end{lem}
\begin{proof}    From the proposition  \ref{Tt}, we have
$$i(\widetilde{m} +k)\mathcal{R}(n,m)= \frac{1}{2\pi i}\int_{-\pi}^{\pi} h(x)\widetilde{T^k}\Big(\mathcal{E}_m(ix)\Big)\mathcal{E}_n^k(-ix) \delta_{k}(x)dx.$$
Integrating by parts  yields
\begin{eqnarray*}
  && \int_{-\pi}^{\pi}h(x)\Big(\mathcal{E}_m^k(ix)\Big)'\mathcal{E}_n^k(-ix)\delta_k(x)dx=  \\&& \qquad\qquad \frac{i}{2}\int_{-\pi}^{\pi}h(x)\mathcal{E}_m^k(ix)\mathcal{E}_n^k(-ix)\delta_k(x)dx
-\int_{-\pi}^{\pi}h(x)\mathcal{E}_m^k(ix)\Big(\mathcal{E}_n^k(-ix)\Big)' \delta_k(x)dx\\
&&\qquad\qquad-2k \int_{-\pi}^{\pi}h(x)\mathcal{E}_m^k(ix)\mathcal{E}_n^k(-ix)\cot(x)\delta_k(x)dx.
\end{eqnarray*}
On the other hand, one can make the following
\begin{eqnarray*}
&&2ki\int_{-\pi}^{\pi}h(x)\frac{\mathcal{E}_m^k(ix)-\mathcal{E}_m^k(-ix)}{1-e^{-2ix}}\mathcal{E}_n^k(-ix)\delta_k(x)dx
\\&&\qquad\qquad\qquad= 2ki\int_{-\pi}^{\pi}\mathcal{E}_m^k(ix)\left\{ \frac{h(x)\mathcal{E}_n^k(-ix) }{1-e^{-2ix}}
- \frac{h(-x)\mathcal{E}_n^k(ix) }{1-e^{2ix}}\right\} \delta_k(x)dx
\\&&\qquad\qquad\qquad= 2k\int_{-\pi}^{\pi}h(x)\mathcal{E}_m^k(ix)\mathcal{E}_n^k(-ix)\cot( x)\delta_k(x)dx
\\&&\qquad\qquad\qquad\qquad\qquad+ 2ki \int_{-\pi}^{\pi}h(x)\mathcal{E}_m^k(ix)
 \frac{ \mathcal{E}_n^k(-ix)-\mathcal{E}_n^k(ix) }{1-e^{2ix}}  \delta_k(x)dx
\\&&\qquad\qquad\qquad\qquad\qquad-2k\int_{-\pi}^{\pi}e^{-ix}\mathcal{E}_m^k(ix)\mathcal{E}_n^k(ix) \cos (x/2) \delta_{k-1/2}(x)dx.
 \end{eqnarray*}
 Therefore we obtain
\begin{eqnarray*}
&& \int_{-\pi}^{\pi}h(x) \widetilde{T^k}\Big(\mathcal{E}_m^k(ix)\Big)\mathcal{E}_n^k(-ix)\delta_k(x)dx =-
\int_{-\pi}^{\pi}h(x)\mathcal{E}_m^k(ix)
 \overline{ \widetilde{T^k}\Big(\mathcal{E}_n(ix)\Big)}\Lambda_k(x)dx-
  \\&&2k\int_{-\pi}^{\pi}e^{-ix}\mathcal{E}_m^k(ix)\mathcal{E}_n^k(ix)\cos(x/2)\delta_{k-1/2}(x)dx
 +\frac{i}{2}\int_{-\pi}^{\pi}h(x)\mathcal{E}_m^k(ix)\mathcal{E}_n^k(-ix)\delta_k(x)dx
\end{eqnarray*}
and, since
$$\int_{-\pi}^{\pi}h(x)\mathcal{E}_m^k(ix)
 \overline{ \widetilde{T^k}\Big(\mathcal{E}_n(ix)\Big)}\Lambda_k(x)dx+
= -i(\widetilde{n}+k) \mathcal{R}(m,n)$$
then   the lemma follows.
\end{proof}
  As an immediately  consequence of Lemma \ref{l1} we have the following estimate.
\begin{cor}
For all $m,n\in \mathbb{Z}$, such that $n\neq  m$ we have
\begin{equation}\label{kmn}
     |\mathcal{R}(n,m)|\leq \frac{C}{|n-m|},
\end{equation}
for some constant $C$.
\end{cor}
\begin{lem}
The  Riesz kernel $\mathcal{R}$ satisfies
$$|\mathcal{R}(n,m+1)-\mathcal{R}(n,m)|\leq \frac{C}{||m|-|n||^2}, \qquad \frac{|n|}{2}\leq  |m|\leq \frac{3|n|}{2}.$$
\end{lem}
\begin{proof}
Let us write
\begin{eqnarray*}
&& \int_{-\pi}^{\pi}e^{-ix}\mathcal{E}_m^k(ix)\mathcal{E}_n^k(ix) \cos (x/2) \delta_{k-1/2}(x)dx
\\&=&\int_{-\pi}^{\pi}(e^{-ix}-1) \cos (x/2) \mathcal{E}_m^k(ix)\mathcal{E}_n^k(ix)  \delta_{k-1/2 }(x)dx+\int_{-\pi}^{\pi}\mathcal{E}_m^k(ix)\mathcal{E}_n^k(ix) \cos (x/2)\delta_{k-1/2}(x)dx
\\&=& -i\int_{-\pi}^{\pi} h(x) \mathcal{E}_m^k(ix)\mathcal{E}_n^k(ix)  \delta_{k }(x)dx+\int_{-\pi}^{\pi}\mathcal{E}_m^k(ix)\mathcal{E}_n^k(ix) \cos (x/2)\delta_{k-1/2}(x)dx
\end{eqnarray*}
From Lemma \ref{l1} and the formula (\ref{eE}),
 $$(\widetilde{m+1}-\widetilde{n}-1/2)\mathcal{R}( n,m+1)=\frac{k}{\pi}\int_{-\pi}^{\pi}e^{-ix} \mathcal{E}_{m+1}^k(ix)\mathcal{E}_n^k(ix) \cos (x/2)\delta_{k-1/2}(x)dx$$
$$=\frac{k\alpha_m}{\pi} \int_{-\pi}^{\pi} \mathcal{E}_{m}^k(ix)\mathcal{E}_n^k(ix) \cos (x/2)\delta_{k-1/2}(x)dx+
  \frac{k\beta_m}{\pi} \int_{-\pi}^{\pi}\mathcal{E}_{-m}^k(ix)\mathcal{E}_n^k(ix) \cos (x/2)\delta_{k-1/2}(x)dx.$$
We then get
\begin{eqnarray*}
&&(\widetilde{m+1}-\widetilde{n}-1/2)\mathcal{R}(m+1,n)-(\widetilde{m}-\widetilde{n}-1/2)\mathcal{R}(m,n)
\\&=& \frac{k(\alpha_m-1)}{\pi}\int_{-\pi}^{\pi} \mathcal{E}_{m}^k(ix)\mathcal{E}_n^k(ix) \cos (x/2)\delta_{k-1/2}(x)dx
\\&&+ \frac{k\beta_m}{\pi} \int_{-\pi}^{\pi}\mathcal{E}_{-m}^k(ix)\mathcal{E}_n^k(ix) \cos (x/2)\delta_{k-1/2}(x)dx+
i\frac{k}{\pi}\int_{-\pi}^{\pi} h(x)  \mathcal{E}_m(ix)^k\mathcal{E}_n^k(ix)   \delta_{k }(x)dx.
\end{eqnarray*}
It follows that
\begin{eqnarray*}
&&(\widetilde{m+1}-\widetilde{n}-1/2)\Big(\mathcal{R}( n,m+1)-\mathcal{R}(n,m+1)\Big)=(\widetilde{m}-\widetilde{m+1} )\mathcal{R}(m,n)
\\&& + \frac{k(\alpha_m-1)}{\pi}\int_{-\pi}^{\pi} \mathcal{E}_{m}^k(ix)\mathcal{E}_n^k(ix) \cos (x/2)\delta_{k-1/2}(x)dx
\\&&+ \frac{k\beta_m}{\pi} \int_{-\pi}^{\pi}\mathcal{E}_{-m}^k(ix)\mathcal{E}_n^k(ix) \cos (x/2)\delta_{k-1/2}(x)dx-
i\frac{k}{\pi}\int_{-\pi}^{\pi} h(x)  \mathcal{E}_m^k(ix)\mathcal{E}_n^k(ix)   \delta_{k }(x)dx.
\end{eqnarray*}
From   (\ref{kmn}) we have
$$|(\widetilde{m}-\widetilde{m+1} )\mathcal{R}(n,m)|\leq (2k+1)|\mathcal{R}(n,m)|\leq \frac{C}{|m-n|}$$
and in view of (\ref{alpha})
$$|\alpha_m-1|= \frac{k^2}{(|m|+k)\Big(\sqrt{|m|(|m|+k)}+|m|+k\Big)} \leq \frac{C}{|m|}\leq \frac{C}{||n|-|m||}$$
and
$$|\beta_m|\leq \frac{1}{|m|} \leq  \frac{C}{||n|-|m||}, $$
since for  $|n|/2 \leq |m|\leq 3|n|/2 $, we have that  $||n|-|m||\leq |n|/2\leq |m|$.
Moreover, by noting  that
$$ \int_{-\pi}^{\pi} h(x)  \mathcal{E}_m^k(ix)\mathcal{E}_n^k(ix)   \delta_{k }(x)dx= \overline{\mathcal{R}(n,-m+1)}$$
so,   by Lemma \ref{l1} one can  obtain
 $$|\mathcal{R}(n,-m+1)|\leq \frac{C}{|n+m|}\leq \frac{C}{||n|-|m||}.$$
Finally as
$$\frac{1}{|\widetilde{m+1}-\widetilde{n}-1/2|}\leq \frac{C}{|n-m|}\leq \frac{C}{||n|-|m||}$$
the lemma is concluded.
\end{proof}
\begin{thm}
  The kernel $K$ satisfies
$$|K(n,\ell)-K(n,m)|\leq C \frac{|\ell-m|}{||m|-|n||^2},$$
For all $n,m\in \mathbb{Z}$ such that  $ \frac{|n|}{2}\leq  |m|,|\ell|\leq \frac{3|n|}{2}$ and $||n|-|m||\geq 2|m-\ell|$.
\end{thm}
\begin{proof}
Assume that  $|m|<|\ell|$.
  Let us first making the following observation.
 Considering the conditions:  $||n|-|m||\geq 2|m-\ell|$  and   $ \frac{|n|}{2}\leq  |m|,|\ell|\leq \frac{3|n|}{2}$
 the numbers $ m $ and $ \ell $ must have the same sign. Then for all integer $h$ between $m$ and $\ell$
$$||n|-|h||\geq   ||n|-|\ell||\geq \frac{||n|-|m||}{2}. $$
 Consequently,
\begin{eqnarray*}
 |\mathcal{R}(n,\ell)-\mathcal{R}(n,m)|&\leq& \sum_{|m|\leq |h|\leq  |\ell|-1} |K(n, h+1)-K(n,h)|
\\&\leq &C \sum_{|m|\leq |h|\leq |\ell|-1} \frac{1}{||n|-|h||^2}
\\&\leq & C \frac{|m-\ell|}{||n|-|m||^2},
\end{eqnarray*}
 which is the desired estimate.
\end{proof}
\par Now  we can apply  the same arguments   to get  the estimate,
\begin{eqnarray*}
 |\mathcal{R}(m,n)-\mathcal{R}(\ell,n)| \leq  C \frac{|m-\ell|}{||n|-|m||^2}.
\end{eqnarray*}
As the Riesz transforms are bounded on $\ell^2(\mathbb{Z})$ we can then invoke the theorem \ref{ttmm} and
  finally we  state the following.
\begin{thm}
  The generalized discrete Riesz transforms $\mathcal{R}$ and $\mathcal{R}^*$ are bounded operators on $\ell^p(\mathbb{Z})$ for all
$1<p<\infty.$
\end{thm}

 \end{document}